\documentclass[11pt,final]{amsart}
\usepackage{amscd,amsmath,amssymb,hyperref,xypic}
\usepackage{showkeys}
\newtheorem{thm}{Theorem}[section]
\newtheorem{cor}[thm]{Corollary}
\newtheorem{lem}[thm]{Lemma}
\newtheorem{prop}[thm]{Proposition}
\theoremstyle{definition}
\newtheorem{defn}[thm]{Definition}
\theoremstyle{remark}
\newtheorem{rem}[thm]{Remark}
\newtheorem{example}[thm]{Example}
\numberwithin{equation}{section}

\newcommand{\To}{\longrightarrow}

\newcommand{\Sp}{\mbox{Sp}(2n)}
\newcommand{\SO}{\mbox{SO}(2n)}
\newcommand{\PSp}{\mbox{PSp}(2n)}
\newcommand{\PSO}{\mbox{PSO}(2n)}

\newcommand{\gm}{{\mathbb G}_m}

\newcommand{\bun}[1]{{\rm Bun}_{#1}}

\newcommand{\sym}{\text{Sym}}

\newcommand{\F}{{\mathcal F}}
\newcommand{\G}{{\mathcal G}}

\newcommand{\struct}{{\mathcal O}}

\newcommand{\orb}{\mathrm{orb}}

\newcommand{\spec}{{\rm Spec}}
\newcommand{\tildearrow}{\stackrel{\sim}{\longrightarrow}}
\newcommand{\ZZ}{{\mathbb Z}}

\newcommand{\sG}{{\mathfrak G}}

\begin{document}
\title[Period-Index Problem]{The Period-Index Problem of the Canonical Gerbe of 
Symplectic and Orthogonal Bundles}

\author{Indranil Biswas}
\address{Tata Institute for Fundamental Research, Homi Bhabha Road, Mumbai 400005, INDIA}
\email{indranil@math.tifr.res.in}

\author{Emre Coskun}
\address{Tata Institute for Fundamental Research, Homi Bhabha Road, Mumbai 400005, INDIA}
\email{ecoskun@math.tifr.res.in}

\author{Ajneet Dhillon}
\address{Department of Mathematics, Middlesex College, University of Western Ontario, London, ON N6A 5B7 CANADA}
\email{adhill3@uwo.ca}



\date{\today}

\begin{abstract}
We consider regularly stable parabolic symplectic and orthogonal bundles
over an irreducible smooth projective curve over an algebraically 
closed field of characteristic zero. The morphism from
the moduli stack of such bundles to its coarse moduli space is a $\mu_2$-gerbe.
We study the period and index of this gerbe, and solve the corresponding period-index problem.
\end{abstract}

\maketitle

\section{Introduction}
Let $X$ be a smooth, projective, irreducible curve over an 
algebraically closed field $k$ of characteristic zero. 
A symplectic (respectively, orthogonal) vector bundle with parabolic structure on 
$X$ consists of a parabolic vector bundle $(E_*,B)$ with a skew-symmetric
(respectively, symmetric) pairing
$$
E_* \otimes E_* \longrightarrow \struct_X,
$$
satisfying a nondegeneracy condition (see Section \ref{sec-moduli}). Note that the above tensor
product above is a tensor product in the category of parabolic vector bundles. 

A regularly stable symplectic (respectively, orthogonal) parabolic bundle is one 
whose automorphism group coincides with the center of the
symplectic (respectively, orthogonal) group. Let $\bun{G}^{\rm rs}$ be the
moduli stack of regularly stable symplectic or orthogonal parabolic bundles, and
let $M^{\rm rs}_G$ be the corresponding coarse moduli space.
We have a $\mu_2$-gerbe
$$
\bun{G}^{\rm rs} \,\longrightarrow\, M^{\rm rs}_G\, .
$$
The purpose of this paper is to study the period-index problem for this gerbe. 
The center of $\text{SO}(2n+1)$ is trivial, hence in this case
$\bun{G}^{\rm rs}\,=\, M^{\rm rs}_G$. Therefore, we will assume that
the rank in the orthogonal case is even.

The index is computed in Theorem \ref{t:theindex}. The main idea is to degenerate 
the gerbe to a highly singular point (see Proposition \ref{prop-degen})
and to use Luna's \'etale slice theorem to study the geometry of this stack over 
the moduli space here.

In Section \ref{sec-moduli}, we briefly recall the definition and properties of 
parabolic bundles on $X$, and recall the construction of their moduli spaces. In 
Proposition \ref{prop-p} we compute the period of the canonical gerbe in most 
cases. We also discuss an application of Luna's \'etale slice theorem to our 
case using Proposition \ref{prop-degen}. In Section \ref{sec-twisted}, after 
a brief overview of twisted sheaves, we give an upper bound for the 
index of the canonical gerbe.

In Sections \ref{sec-stable-triple} and \ref{sec-fieldvalued}, we discuss the 
concept of a stable central simple algebra with involution, and prove that for 
a stable central simple algebra with 
involution over a field $F$, there exists a morphism ${\rm Spec}\,F 
\longrightarrow Z^s/G^{ad}$, where 
$G$ denotes an appropriate symplectic or orthogonal group. Finally, in Section 
\ref{sec-construction}, we prove the existence of lower bounds for the index of 
the canonical gerbe. This gives the index completely in 
the symplectic case, and it gives a very strong lower bound for the orthogonal 
case. The appendices discuss a technical result used in the proof, and 
symplectic or orthogonal involutions on a central simple algebra.

\subsection{Conventions}\label{se-con}
\begin{itemize}
 \item We work over an algebraically closed base field $k$ of characteristic 
zero.

 \item By $G$ we denote $\Sp$ or $\SO$, and $G^{ad}$ denotes the adjoint form of 
$G$, 
i.e. $\PSp$ or $\PSO$. In the case of $G=\SO$, we assume that $n \geq 2$. 
Also, 
$G(j)$ denotes $\mbox{Sp}(j)$ or $\mbox{SO}(j)$. By $\mathfrak{g}$ we denote the Lie algebra of $G$, while $\mathfrak{g}(j)$ denotes the Lie algebra of $G(j)$.

 \item $X$ denotes an irreducible smooth projective curve of genus 
$g(X) \geq 2$.

 \item $D := \{x_1 ,\cdots ,x_n\} \subset X$ is the ordered set of 
parabolic 
points on $X$.

 \item $M^{\rm rs}_{G}$, $M^{\rm s}_{G}$ and $M^{\rm ss}_{G}$ denote 
the coarse moduli spaces of regularly stable, stable and semistable 
parabolic $G$-bundles, while $\text{Bun}^{\rm rs}_{G}$, $\text{Bun}^{\rm 
s}_{G}$ and $\text{Bun}^{\rm ss}_{G}$ are similarly defined stacks of parabolic $G$-bundles.

 \item $\mathcal{E}$ is the universal parabolic bundle on $X \times 
\text{Bun}^{\rm *}_{G}$ where ``$*$'' stands for $rs$, $s$ or $ss$.
 
 \item $p: X \times \text{Bun}^{\rm rs}_{G} \To \text{Bun}^{\rm rs}_{G}$ denotes 
the natural projection.
\end{itemize}

\section{The Moduli Space of Parabolic Bundles}\label{sec-moduli}

\subsection{Parabolic $G$-Bundles over a Curve}

Let $E_*$ be a parabolic vector bundle on $X$ with parabolic structure 
over $D$. A \emph{bilinear form} on $E_*$ is a homomorphism of parabolic 
bundles $B : E_*\otimes E_*  \longrightarrow {\mathcal O}_X$, where 
${\mathcal O}_X$ is the trivial line bundle with trivial parabolic 
structure (this means that there are no nonzero parabolic weights). The 
parabolic 
vector bundle $E_*\otimes E^\vee_*$, where $E^\vee_*$ is the parabolic dual,
is given by the sheaf of endomorphisms of 
$E_*$ compatible with the parabolic structure. We have a natural homomorphism of 
parabolic bundles $h : {\mathcal O}_X \longrightarrow E_*\otimes E^\vee_*$ (as 
before, ${\mathcal O}_X$ is the trivial line bundle with trivial parabolic 
structure) that sends any locally defined function $f$ to the locally defined 
endomorphism $s \longmapsto f\cdot s$ of $E_*$. Given a bilinear form $B$ on 
$E_*$, the composition
\begin{equation}\label{f1}
 E_*=E_*\otimes {\mathcal O}_X \stackrel{{\rm Id}\otimes h}{\longrightarrow} 
E_*\otimes E_*\otimes E^\vee_* \stackrel{B \otimes {\rm Id}}{\longrightarrow} 
E^\vee_*
\end{equation}
will be denoted by $\widehat{B}$.

A \emph{symplectic} parabolic vector bundle is a pair $(E_*, B)$, where $E_*$ is 
a parabolic vector bundle, and $B$ is a skew-symmetric bilinear form on $E_*$ 
such that the above homomorphism $\widehat{B}$ is an isomorphism. A symplectic 
parabolic vector bundle $(E_*, B)$ of rank $2n$ will be called a 
\textit{parabolic} $\Sp$-bundle. An \emph{orthogonal} parabolic vector bundle 
$(E_*, B)$ is defined similarly, with a symmetric bilinear form $B$ on $E_*$. An 
orthogonal parabolic vector bundle $(E_*, B)$ of rank $n$ will be called a 
\textit{parabolic} ${\rm SO}(n)$-bundle.

If $\widehat{B}$ is an isomorphism, then $\text{\rm par-deg}(E_*) = \text{\rm 
par-deg}(E^\vee_*) = - \text{\rm par-deg}(E_*)$. Hence symplectic and orthogonal 
parabolic bundles have parabolic degree zero.

An \textit{isomorphism} between two parabolic $G$-bundles $(E_*,B)$ and 
$(E'_*,B')$ is an isomorphism of parabolic vector bundles $\phi : E_* 
\longrightarrow E'_*$ such that the following diagram commutes:
\[
\begin{matrix}
 E_*\otimes E_* & \stackrel{\phi\otimes \phi}{\longrightarrow} & 
 E'_*\otimes E'_*\\
 ~\Big\downarrow B && ~\Big\downarrow B'\\
 {\mathcal O}_X & \stackrel{{\rm Id}}{\longrightarrow} &
 {\mathcal O}_X
\end{matrix}
\]

A parabolic $G$-bundle $(E_*,B)$ is called \textit{stable} (respectively, \textit{semistable}) if for any nonzero proper parabolic subbundle $F_* \subset E_*$ with $B(F_*\otimes F_*)=0$, the inequality
$$
\frac{\text{\rm par-deg}(F_*)}{{\rm rank}(F_*)} <
\frac{\text{\rm par-deg}(E_*)}{{\rm rank}(E_*)} ~{\rm 
(respectively,~} \frac{\text{\rm par-deg}(F_*)}{{\rm rank}(F_*)} 
\leq\, \frac{\text{\rm par-deg}(E_*)}{{\rm rank}(E_*)}{\rm )}
$$
holds.

A stable parabolic $G$-bundle $(E_*,B)$ is called \textit{regularly stable} if it has no automorphism other than $\pm \text{Id}_{E_*}$.

\begin{rem}\label{r:structure}
Let $(E_*, B)$ be a parabolic ${\rm SO}(a)$-bundle. Then for each positive 
integer $b$, the parabolic vector bundle $(E_*, B)^{\oplus b}$ has a natural 
structure of a parabolic ${\rm SO}(ab)$-bundle; and if $b$ is even, then $(E_*, 
B)^{\oplus b}$ also has the structure of an ${\rm Sp}(ab)$-bundle. Indeed, 
$(E_*, B)^{\oplus b}\, =\, (E_*, B)\otimes k^{\oplus b}$. Now put the standard 
symplectic or orthogonal structure on $k^{\oplus b}$. The orthogonal structure 
$B$ on $E_*$ and the symplectic or orthogonal structure on $k^{\oplus b}$ 
together define a symplectic or orthogonal structure on the tensor product $(E_*, 
B)\otimes k^{\oplus b}$.
\end{rem}

\subsection{A Construction of the Moduli Space}
\label{ss:moduli}

Fix a parabolic type, meaning parabolic weights and quasi-parabolic filtration 
types. We assume that the parabolic type is so chosen that
there is a parabolic $G$-bundle of the given type. This requires the
parabolic type to be compatible with the $G$-structure of the bundles.
Let $M^{\rm ss}_{G}$ be the moduli space of semistable parabolic $G$-bundles 
of the given parabolic type.

We now briefly describe the construction of this moduli space $M^{\rm ss}_{G}$ as 
done in \cite{BBN01}. Given a parabolic type, there is an associated 
Galois cover $Y\longrightarrow X$ with finite Galois group $\Gamma$. Now 
\cite[Theorem 4.3]{BBN01} 
gives an equivalence between parabolic $G$-bundles on $X$ and $(\Gamma,G)$-bundles on $Y$. 
Using this equivalence, the moduli space of semistable parabolic $G$-bundles on $X$ is 
constructed in the following manner. Choose an integer $m_0$ such that for $m \geq m_0$
 and any semistable $\Gamma$-bundle $\mathcal{E}$ on $Y$ of rank $2n$ and trivial determinant
 one has $h^i(\mathcal{E}(m))=0$ for $i > 0$ and $\mathcal{E}(m)$ is globally generated. 
Let $N=h^0(\mathcal{E}(m))$. Let $P(t)=2n(t+1-g(X))$ be the Hilbert polynomial of $\mathcal{E}$.
 Then there is a well-known Quot scheme $Q$, constructed by Grothendieck, which parametrizes 
quotients of $\mathcal{O}_{Y}(-m)^N$ with Hilbert polynomial $P$. The finite group $\Gamma$ 
acts on $Q$. Let $\G$ denote the group of $\Gamma$-invariant automorphisms of 
$\mathcal{O}_{Y}(-m)^N$. This is a reductive group by \cite{csS69}. 
Let $Q^\Gamma$ denote the $\Gamma$-invariant locus in $Q$.
There exists a nonempty open subset $\mathcal{R}^{ss}$ of $Q^\Gamma$ consisting 
of semistable bundles. One then constructs a scheme $Q_G$ with a $\G$-action 
together with a $\G$-equivariant morphism $Q_G\longrightarrow \mathcal{R}^{ss}$. 
Then the 
GIT quotient of $Q_G$ by $\G$ is the moduli space $M^{\rm ss}_{G}$. We note that 
there exist open subsets $\mathcal{R}^{rs} \subset \mathcal{R}^{s} \subset 
\mathcal{R}^{ss}$ consisting of regularly stable and stable parabolic 
$G$-bundles; and their GIT quotients under $\mathcal{G}$, $M^{\rm rs}_{G} 
\subset M^{\rm s}_{G} \subset M^{\rm ss}_{G}$ are the corresponding coarse 
moduli spaces. (We refer the reader to \cite[Section 5]{BBN01} for the 
details of this construction.)

\subsection{The period of the canonical gerbe}

Let $d$ be the degree of the vector bundle underlying a parabolic $G$-bundle in $M^{\rm rs}_{G}$. Then the total parabolic weight is $-d$ since the parabolic degree is zero. At $x_i \in D$, let $\{n_{1,i}, \cdots, n_{\ell_i,i}\}$ be the multiplicities of the parabolic weights at $x_i$. (Recall that these are the dimensions of the graded pieces of the quasi-parabolic filtration at $x_i$.)

Define
\begin{equation}\label{dede}
 \epsilon := \text{g.c.d.}\{d, 2n, \{n_{1,i}, \cdots, 
n_{\ell_i,i}\}_{i=1}^n\}.
\end{equation}

\begin{prop}\label{prop-p}
If $\epsilon$ is odd, there is a Poincar\'e vector bundle on $X \times M^{\rm rs}_{G}$. If $\epsilon$ is even, there is no Poincar\'e vector bundle on $X \times M^{\rm rs}_{{\rm Sp}(2n)}$. If $\epsilon \geq 4$ is even, there is no Poincar\'e vector bundle on $X \times M^{\rm rs}_{{\rm SO}(2n)}$. If $\epsilon = 2$, there is no Poincar\'e vector bundle on $X \times M^{\rm rs}_{{\rm SO}(4n)}$.
\end{prop}

\begin{proof}
Recall that $\text{Bun}^{\rm rs}_{G}$ is the moduli stack of regularly stable parabolic $G$-bundles, and ${\mathcal E} \longrightarrow X \times \text{Bun}^{\rm rs}_{G}$ is the universal parabolic bundle. Let
\begin{equation}\label{f}
 f : \text{Bun}^{\rm rs}_{G} \longrightarrow M^{\rm rs}_{G}
\end{equation}
be the morphism to the coarse moduli space.

For any integer $m$, its image in ${\mathbb Z}/2\mathbb Z$ will be 
denoted by $\overline{m}$. Note that for each integer $n_{\ell,i}$ in 
\eqref{dede}, there is a line bundle of weight $\overline{n_{\ell,i}}$ 
on $\text{Bun}^{\rm rs}_{G}$. Similarly, there is a line bundle of 
weight $\overline{d+2n(1-g(X))}\,=\, \overline{d}$ on $\text{Bun}^{\rm 
rs}_{G}$ given by the determinant of cohomology of ${\mathcal E}$.

First assume that $\epsilon$ is odd. Then there is a line bundle $L$ of weight 
$\overline{1}$ on $\text{Bun}^{\rm rs}_{G}$. Hence the vector bundle ${\mathcal 
E}\otimes p^*(L)$ on $X \times \text{Bun}^{\rm rs}_{G}$ has weight zero; 
therefore, it descends to $X \times M^{\rm rs}_{G}$ as a Poincar\'e vector bundle.

Now assume that $\epsilon$ is even; if $G=\SO$, then assume that
$\epsilon \geq 4$. We will show that there is no Poincar\'e vector bundle on $X 
\times M^{\rm rs}_{G}$.

Let us first consider the $\Sp$ case. If $(V_*, B_0)$ is a parabolic $\text{SO}(n)$-bundle, then the parabolic direct sum $V_*\oplus V_* = V_*\otimes k^{\oplus 2}$ has a symplectic structure. Indeed, the symmetric bilinear form $B_0$ on $V_*$ and the standard symplectic form on $k^{\oplus 2}$ together define a parabolic $\Sp$-structure $\widetilde{B}_0$ on $V_*\otimes k^{\oplus 2}$.
Since $\epsilon$ is a multiple of $2$, there is stable parabolic 
$\text{SO}(n)$-bundle $(V_*, B_0)$ such that the corresponding parabolic 
$\Sp$-bundle $(V_*\otimes k^{\oplus 2}, \widetilde{B}_0)$ is a
parabolic $\Sp$-bundle of the given type; this parabolic $\Sp$-bundle is 
semistable because $(V_*, B_0)$ is semistable.

Now consider the $\SO$ case. We can then go through the same 
construction as follows. We choose a stable parabolic 
$\text{SO}(2n/\epsilon)$-bundle 
$(V_*, B_0)$ such that the corresponding parabolic $\SO$-bundle $(V_*\otimes 
k^{\oplus \epsilon}, \widetilde{B}_0)$ is a semistable parabolic $\SO$-bundle of 
the given type; here $k^{\oplus \epsilon}$ is equipped with the standard 
orthogonal form.

Finally, consider ${\rm SO}(4n)$ with $\epsilon =2$. In this case there is a 
stable parabolic $\Sp$-bundle $(V_*, B_0)$ such that $(V_*\otimes k^{\oplus 2},  
\widetilde{B}_0)$ is a semistable parabolic ${\rm SO}(4n)$-bundle, where 
$k^{\oplus 2}$ is equipped with the standard symplectic form.

The automorphism group of $(V_*\otimes k^{\oplus 2}, \widetilde{B}_0)$ in the 
$\Sp$ case, or $(V_*\otimes k^{\oplus\epsilon}, \widetilde{B}_0)$ and 
$(V_*\otimes k^{\oplus 2}, \widetilde{B}_0)$ in the two $\SO$ cases, contains 
$\text{Sp}(2)$ or $\text{SO}(\epsilon)$, and hence the center $\pm\text{Id}$ of 
this $\text{Sp}(2)$ or $\text{SO}(\epsilon)$ is contained in the automorphism 
group. Since $\text{Sp}(2)$ or $\text{SO}(\epsilon)$ does not have any nontrivial 
character, by \cite[p. 1286, Theorem 2.2]{BH} we conclude that there is no 
Poincar\'e vector bundle on $X \times M^{\rm rs}_{G}$.
\end{proof}

\begin{rem}\label{rem-period2}
If $\epsilon$ is odd, Proposition \ref{prop-p} shows that the canonical 
$\mu_2$-gerbe $\text{Bun}^{\rm rs}_{G}\longrightarrow M^{\rm rs}_{G}$ is 
neutral. 
Hence its period and index are 1. Unless otherwise stated, from now on 
we assume that $\epsilon$ is even and that $\epsilon \geq 4$ in the 
$\SO$ case (note that this condition is automatically satisfied if 
there are no parabolic points).
\end{rem}

\begin{prop}\label{prop-degen}
Let $r = 2n/\epsilon$. There is a regularly stable parabolic ${\rm
SO}(r)$-bundle 
$(E_*, B)$ such that $(E_*, B)^{\oplus \epsilon} \in M^{\rm ss}_{{\rm SO}(2n)}$.
Similarly, there is a regularly stable parabolic ${\rm SO}(r)$-bundle $(E_*, B)$
such that $(E_*, B)^{\oplus \epsilon} \in M^{\rm ss}_{{\rm Sp}(2n)}$.
\end{prop} 

\begin{proof}
Fix rank, degree, quasi-parabolic types and parabolic weights. The necessary and sufficient condition for the existence of a parabolic orthogonal bundle with this data is the following:
\begin{enumerate}
 \item The parabolic degree is zero,
 \item At each parabolic point, if $\pi$ is a parabolic weight, then $1-\pi$ is also a parabolic weight, and
 \item At each parabolic point, the multiplicity of any parabolic weight $\pi$ coincides with the multiplicity of the parabolic weight $1-\pi$.
\end{enumerate}

The necessary and sufficient condition for the existence of a parabolic symplectic bundle with this data is the following:
\begin{enumerate}
 \item The rank is even,
 \item The parabolic degree is zero,
 \item At each parabolic point, if $\pi$ is a parabolic weight, then $1-\pi$ is also a parabolic weight, and
 \item At each parabolic point, the multiplicity of any parabolic weight $\pi$ coincides with the multiplicity of the parabolic weight $1-\pi$.
\end{enumerate}

If there is a parabolic orthogonal (respectively, symplectic) bundle, 
then there is a semistable parabolic orthogonal (respectively, 
symplectic) bundle. This follows from the fact that the stratum of 
parabolic orthogonal or symplectic bundles with given Harder-Narasimhan 
filtration type has dimension less than the dimension of the moduli 
space. Again for dimension reasons, there is a regularly stable bundle
$(E_*\, ,B)$ as in the statement of the proposition.
\end{proof}

\subsection{Luna's \'Etale Slice Theorem}

We follow the exposition in \cite{jmD04}. 

\begin{defn}
Let $H$ be a reductive linear algebraic group. An $H$-equivariant morphism $S 
\longrightarrow T$ of varieties is said to be \emph{strongly \'{e}tale} if 
$S/\!\!/H \longrightarrow T/\!\!/H$ is \'{e}tale.
\end{defn}

\begin{thm}\label{thm-les}
Suppose that $H$ acts on a smooth quasi-projective variety $S$ and the action 
is linearized with respect to some very ample line bundle. Let $s\in S$ be a 
closed point with stabilizer $H_s$ and closed orbit. Then there is an 
$H$-stable open subset $U \subseteq S$, containing $s$, and $V\subseteq U$ a 
$H$-stable smooth subvariety, such that if $N_s$ is the normal space to the 
orbit $H\cdot s$ at $s$,
then we have an equivariant diagram of strongly \'etale morphisms
\[
\xymatrix{
    &  H \times_{H_s} V \ar[dr] \ar[dl] & \\
 U  &                  & W\subseteq H \times_{H_s} N_s \\
}
\]
\end{thm}

\begin{proof}
See \cite[page 27]{jmD04}.
\end{proof}

\begin{example}\label{ex-slice}
We apply Theorem \ref{thm-les} to $Q_G$. For $s$, we take a point corresponding 
to $(E_*, B)^{\oplus \epsilon}\,=\, (E_*, B)\otimes k^{\oplus \epsilon}$ as in 
Proposition \ref{prop-degen}. 
Recalling notation from Section \ref{se-con}, we have 
$H_s=G(\epsilon)$.
Take $V\subseteq Q_G$ as in the theorem. 

Recall the discussion at the start of Section \ref{ss:moduli}.
There is a diagram
\begin{equation}
\label{e:diagram}
 \xymatrix{
U \ar[d] & \G\times_{H_s} V \ar[l] \ar[r] \ar[d] & W \subseteq \G\times_{H_s} N_s \ar[d] \\
[U/\G]   \ar[d] &  [V/H_s] \ar[r] \ar[l] \ar[d]   & [W/\G] \subseteq [N_s/H_s] \ar[d] \\
U/\!\!/\G &   V/\!\!/H_s \ar[r]_{\pi_U} \ar[l]^{\pi_W} & W/\!\!/\G \subseteq 
N_s/\!\!/H_s
}
\end{equation}
Generically, $U\longrightarrow U/\!\!/\G$ is a $\G/\mu_2$-bundle. As $\mu_2$ 
acts 
trivially on $V$, a
similar statement is true for the middle composite. Generically 
$W\longrightarrow W/\!\!/\G$ is also a
$\G/\mu_2$-bundle; see Proposition \ref{p:normal} below. It follows that the 
middle row consists, generically, of $\mu_2$ gerbes over
the bottom row.
\end{example}

\begin{prop}
$\mathcal{R}^{ss}$ is smooth.
\end{prop}
\begin{proof}
See Remark 5.6 of \cite{BBN01}.
\end{proof}

For a parabolic $G$-bundle $(E_*, B)$, let $\mathcal{E}nd(E_*)$ be the
subsheaf of $\mathcal{E}nd(E)$ defined by the sheaf endomorphisms
preserving the quasi-parabolic filtrations. (So $\mathcal{E}nd(E_*)$
is the vector bundle underlying the parabolic tensor
product $E_*\otimes E^\vee_*$.) Let $$\mathcal{E}nd_B(E_*)
\, \subset\, \mathcal{E}nd(E_*)$$ be the subbundle defined by the
sheaf of endomorphisms $\beta$ such that $B(\beta(v)\otimes w)+
B(v\otimes (w))\,=\, 0$ for all locally defined sections $v$ and $w$
of $E$.

\begin{prop}\label{p:normal}
Consider $s \in \mathcal{R}^{ss}$ as in Example \ref{ex-slice}.
 \begin{enumerate}
  \item The stabilizer of $s$ inside $\G$ is $G(\epsilon)$.
  \item The normal space to $\orb(s)$ at $s$ can be identified with
\[
 \mbox{H}^1(X,\mathcal{E}nd_B(E_*)) \otimes \mathfrak{g}(\epsilon).
\]
  \item The natural action of the stabilizer on the normal bundle can be identified with the adjoint action on the $\mathfrak{g}(\epsilon)$ factor above.
 \end{enumerate}
\end{prop}

\begin{proof}
Recall that $(E_*, B)^{\oplus \epsilon}$ is Example \ref{ex-slice} is regularly stable. In view of the definition of a regularly stable bundle, the proposition follows.
\end{proof}

\begin{rem}\label{rem-Z}
Denote $g\,=\,\mbox{h}^1(X,\mathcal{E}nd_B(E_*))$. Note that $g \geq 2$. We 
write 
$$
Z=Z(\mathfrak{g}(\epsilon),g) = \mathfrak{g}(\epsilon)^{\otimes g}
$$
for the normal space to $\orb(s)$ at $s$ as in Proposition \ref{p:normal}.
This is the normal space that occurs in the diagram \eqref{e:diagram}.

Set $\Lambda=k \langle z_1,z_2,\cdots,z_g \rangle$ be a polynomial ring in $g$ 
non-commuting variables. A closed point of $Z$ determines a $\Lambda$-module 
structure on $k^{\epsilon}$. Denote by $Q$ the standard symplectic or 
orthogonal form on $k^{2n}$. There is an open subset $Z^s$ of $Z$ consisting of 
points where the corresponding $\Lambda$-module has no nontrivial isotropic 
submodules. This coincides with stable locus for the adjoint action of
$G(\epsilon)$, see \cite[Proposition 4.2]{CDL10}. 

Using the notation of \eqref{e:diagram} we define the following loci :
\begin{eqnarray*}
 (V/\!\!/H_s)^t &=& \pi_U^{-1}((U/\!\!/\G)^{\rm rs}) \cap \pi_W^{-1}(Z^s/H_s) \\
 (Z^s/H_s)^t &=& \pi_W((V/\!\!/H_s)^t) \\
 (U/\!\!/\G)^t &=& \pi_U(V/\!\!/H_s)^t)
\end{eqnarray*}
As the two varieties on the corners of the bottom row in \eqref{e:diagram} are 
irreducible, and all maps are \'etale,
these are non-empty open sets.

We have a marked point $0$ of the quotient $U/\!\!/\G$, corresponding to the 
bundle 
in Proposition \ref{prop-degen}. The diagram \eqref{e:diagram} induces an 
isomorphism
$$
\widehat{\struct_{U/\!\!/\G,0}} \,\cong\, 
\widehat{\struct_{Z/\!\!/G(\epsilon),0}}.
$$ 
We will denote this ring by 
$$
\widehat{\struct_{0}}.
$$
Finally we construct an open subscheme ${\rm Spec}(\widehat{\struct_{0}})^t$
of ${\rm Spec}(\widehat{\struct_{0}})$
by the following Cartesian diagram :
\[
 \xymatrix{
 {\rm Spec}(\widehat{\struct_{0}})^t \ar[r] \ar[d] & (V/\!\!/H_s)^t \ar[d] \\
 {\rm Spec}(\widehat{\struct_{0}}) \ar[r] & (V/\!\!/H_s)
}
\]

\end{rem}

\begin{prop}
\label{p:same}
 The classes of the three $\mu_2$-gerbes defined in \eqref{e:diagram}
are the same inside $${\rm Br}({\rm Spec}(\widehat{\struct_{0}})^t)$$.
\end{prop}

\begin{proof}
This follows from \cite[Chapter IV, 2.3.18]{jG71}.
\end{proof}

\section{Twisted Sheaves}\label{sec-twisted}

Consider a gerbe $\sG\longrightarrow S$ banded by a sheaf of abelian groups $A$ 
that is a subgroup
of $\gm$. A coherent sheaf $\F$ on $\sG$ has two actions of $A$ on it, the inertial action
and a second action by viewing $A$ as a subsheaf of $\struct_{\sG}$. A \emph{twisted sheaf}
on $\sG$ is a coherent sheaf where these two actions coincide.

If $\F$ is a locally free twisted sheaf on $\sG$ then $\F\otimes\F^\vee$ descends to a 
Azumaya algebra on $S$ whose Brauer class is the same as the Brauer class of
$\sG$, see the proof of \cite[Proposition 3.1.2.1]{mL08}. 

\begin{example}\label{e:nori}
 The natural action of $G(\epsilon)$ on $k^\epsilon$ produces a twisted sheaf
on $[Z/G(\epsilon)]$. Explicitly, the action of $G(\epsilon)$ on $Z$ extends
to an action on the trivial bundle $Z\times k^\epsilon$. The Azumaya algebra on
$ Z^s/\!\!/G(\epsilon)$ associated to this twisted sheaf will be denoted
$\mathcal{B}$. The corresponding trivial Azumaya algebra on $Z^s$ will be denoted $\mathcal{A}$. Note that $\mathcal{B}$ pulls back to $\mathcal{A}$ under $\pi$.
\end{example}

Twisted sheaves are a useful tool for understanding the difference between the period and the index. Let
us assemble the pertinent results.

\begin{prop}\label{prop-periodindex}
When $S\,=\,\spec(K)$ in the above situation the period divides the index and the 
period and index have the same prime factors.
\end{prop}
\begin{proof}
This is well known; for example, see \cite{DF93}.
\end{proof}

\begin{prop}\label{prop-twisted}
Let $\sG\longrightarrow\spec(K)$ be a $\gm$-gerbe over a field. Then the index of $\sG$ divides $m$ if and only if there is a locally free twisted sheaf on $\sG$ of rank $m$.
\end{prop}

\begin{proof}
See \cite[Proposition 3.1.2.1]{mL08}.
\end{proof}

Now consider the canonical $\mu_2$-gerbe $\text{Bun}^{\rm rs}_{G}\longrightarrow 
M^{\rm 
rs}_{G}$. In Remark \ref{rem-period2}, we have assumed 
that $\epsilon$ is even. Hence the period of the canonical gerbe is 2, and by 
Proposition \ref{prop-periodindex}, its index (over the function field of $M^{\rm rs}_{G}$) is a power of 2.

{}From now on, we will write $\epsilon=2m$.

\begin{prop}\label{prop-indexdivides}
The index of the canonical gerbe $\text{Bun}^{\rm rs}_{G}\longrightarrow M^{\rm 
rs}_{G}$ divides $\epsilon$.
\end{prop}

\begin{proof}
To see that the index of the canonical gerbe divides the $n_{k,i}$ corresponding to a point $x_i$, one considers the restriction of the universal parabolic bundle on $X \times M^{\rm rs}_{G}$ to $x_i \times M^{\rm rs}_{G}$ and takes the graded piece corresponding to $n_{k,i}$.

To complete the proof, we have to produce a twisted sheaf of rank $d$ on 
$\text{Bun}^{\rm rs}_{G}$. For $k \gg 0$, one has $R^1 p_*(\mathcal{E} \otimes 
p^* \mathcal{O}_{X}(k))=0$. Since we are over the regularly stable locus, 
$p_*(\mathcal{E})$ is a twisted sheaf of rank $d+2n(1-g)$. This finishes the proof by Proposition \ref{prop-twisted}.
\end{proof}

\begin{cor}
Assume $G=SO(4n)$ and $\epsilon=2$. Then the period and index of the canonical gerbe are both 2.
\end{cor}
\begin{proof}
By Proposition \ref{prop-p}, there is no Poincar\'e vector bundle on $X \times M^{\rm rs}_{SO(4n)}$. Hence, the period of the canonical gerbe is 2. By Proposition \ref{prop-indexdivides}, the index divides $\epsilon=2$. The result follows.
\end{proof}

\section{Stable central simple algebras with 
involution}\label{sec-stable-triple}

\begin{defn}
Let $A$ be a central simple algebra of degree $(2m)^2$ with involution $\sigma$ over a 
field $F$. An $F$-subalgebra $B \subseteq A$ is called \emph{parabolic} if for 
some finite field extension $K/F$ that splits $(A,\sigma)$, and an isomorphism 
$\phi:A_K \longrightarrow \mbox{End}(K^{2m},Q_0)$, where $Q_0$ is the standard 
symplectic or 
orthogonal form, $B_K$ leaves a (nontrivial) totally isotropic subspace $W 
\subseteq K^{2m}$ invariant.
\end{defn}

\begin{prop}
The definition above is independent of the extension $K/F$ and 
$\phi:A_K\longrightarrow \mbox{End}(K^{2m},Q_0)$.
\end{prop}

\begin{proof}
One easily reduces the question to the following situation: we have a finite 
field extension $L/K$ and a splitting 
$$
\alpha_K : A_K \longrightarrow \mbox{End}(K^{2m},Q_0)
$$
such that upon base extension to $L$ there is an isotropic subspace $W$ of $L^{2m}$ 
preserved by $B_L$. The result now follows from Lemma \ref{lem-aji}.
\end{proof}

\begin{lem}\label{lem-aji}
Let $F_1 / F_2$ be an extension of fields, and let $P \subseteq {\rm GL}(m,
F_2)$ be a subgroup that becomes parabolic under base extension to $F_1$, 
preserving a subspace $V \subseteq F_1^m$ of dimension $m'$. Then there exists a 
subspace $V'\subseteq F_2^m$ of dimension $m'$, base extending to $V$, that is 
left invariant by $P$.
\end{lem}

\begin{proof}
Since $P_{F_1}$ preserves $V$, the action of $P_{F_1}$ on $\mbox{Grass}(m',m)$ 
has a fixed point, which we will denote by $Q_1$. Since this action is
obtained from the action of $P$ on $\mbox{Grass}(m',m)$, it follows that the
action must also have a 
fixed point $Q_2$ which gives $Q_1$ upon base extension. Hence there 
must be a $P$-invariant subspace $V' \subseteq F_2^m$ base extending to $V$.
\end{proof}

We now give the definition of a stable central simple algebra (CSA for
short) with involution. We remind the reader that the number $g \geq 2$ was 
defined in Remark \ref{rem-Z}.

\begin{defn}
Let $A$ be a central simple algebra of degree $(2m)^2$ with involution $\sigma$ 
over a field $F$. Let $x_1, \cdots, x_g \in A$ be elements such that
$\sigma(x_i)=-x_i$. 
The triple $(A,\sigma,x_i)$ is called a \emph{stable central simple algebra with
involution} over $F$ if the $F$-subalgebra of $A$ generated by $x_i$ is not
contained in a  parabolic subalgebra.
\end{defn} 
\begin{rem}
Note that the degree $2m$ of the central simple algebra is implicit in the definition.
\end{rem}

\begin{example}
\label{e:triple}
Recall the construction of the Azumaya algebra $\mathcal{B}$ from Example
\ref{e:nori}. We now describe its construction in more detail to explain how
the
``universal stable CSA with involution'' is formed.

Recall that there is an action of $G(\epsilon)$ on $Z$ by conjugation. Consider 
the split Azumaya algebra of degree $\epsilon$ on $Z$ defined by the algebra of 
endomorphisms of $k^{\epsilon}$. Using the standard symplectic or orthogonal 
bilinear form on $k^{\epsilon}$, one can construct a canonical symplectic or 
orthogonal involution on $End_k(k^{\epsilon})$. We denote this involution by 
$r$. (We refer the reader to Appendix \ref{sec:csa-inv} for the details.) It is 
an easy exercise to show that $r$ descends to $Z^s/G(\epsilon)^{ad}$. Hence we 
get an Azumaya algebra with symplectic or orthogonal involution 
$(\mathcal{B},r)$ on $Z^s/G(\epsilon)^{ad}$.

Finally, we need to construct $g$ sections $x_1, \ldots, x_g$ of $\mathcal{B}$ 
such that, for any field $F$ and any map $\mbox{Spec }F\longrightarrow 
Z^s/G(\epsilon)^{ad}$, 
the pullback of $\mathcal{B}$ along with $r$ and the $x_i$ give a stable CSA 
with involution over $F$. Recall that $Z=\mathfrak{g}(\epsilon)^{\times g}$. We 
define a section $x_i$ of the split Azumaya algebra of degree $\epsilon$ on $Z$ 
by first taking the $i^{th}$ projection $Z\longrightarrow 
\mathfrak{g}(\epsilon)$ and 
composing with the inclusion $\mathfrak{g}(\epsilon)\longrightarrow 
End_k(k^{\epsilon})$. 
Again, one can check that these sections descend to $Z^s/G(\epsilon)^{ad}$ and they define a stable CSA with involution.
\end{example}

\section{The field-valued points of $Z^s/G(\epsilon)^{ad}$}\label{sec-fieldvalued}

In this section, for a field $F$ containing $k$, we describe the $F$-valued points of 
$Z^s/G(\epsilon)^{ad}$ in terms of the stable central simple algebras (CSA) with 
involution defined in Section \ref{sec-stable-triple}. 
Let ${\rm Fields}/k$ denote the category of field extensions of $k$.
Let $$\Phi_1: {\rm Fields}/k \longrightarrow 
(\mbox{Set})$$ be the functor that sends any
$F$ to the set of isomorphism classes of stable CSAs with 
involution over $F$. 
Given a morphism $\mbox{Spec }K \longrightarrow \mbox{Spec }F$, the 
corresponding map 
$\Phi_1(\mbox{Spec }F) \longrightarrow \Phi_1(\mbox{Spec }K)$ is defined by 
pull-back. 
We also consider the  functor of points
 $$\Phi_2: {\rm Fields}/k \longrightarrow (\mbox{Set})$$
 that takes $\mbox{Spec }F$ to the set $\mbox{Mor}(\mbox{Spec }F,Z^s/G(\epsilon)^{ad})$.

\begin{thm}\label{thm-rep}
The two functors $\Phi_1$ and $\Phi_2$ are naturally equivalent.
\end{thm}

\begin{proof}
There is a natural transformation $\alpha: \Phi_2\longrightarrow \Phi_1$ that 
takes any
 ${\rm Spec}\,F \longrightarrow Z^s/G(\epsilon)^{ad}$ to the pullback of 
$\mathcal{B}$ to 
$\mbox{Spec }F$ via this morphism. We shall define another natural 
transformation $\beta: \Phi_1 \longrightarrow \Phi_2$ and prove that $\alpha$ 
and $\beta$ are inverses to each other.

Let $(A, \sigma, x_i)$ be a stable CSA with involution over the field $F$.
 Choose a finite Galois extension $K/F$ splitting $(A,\sigma,x_i)$. Choose an
isomorphism $A_K \tildearrow \mbox{End}(K^{2m},Q_0)$. This gives a $\Lambda
\otimes K$-module structure on $K^{2m}$; and from the definition of a stable CSA
with involution, we have that this module has no nontrivial isotropic
submodules. Hence by \cite[Proposition 4.2]{CDL10}, we obtain a map $\phi:
{\rm Spec}\, K \longrightarrow Z^s$. Consider the composition $\pi \circ \phi: 
{\rm Spec}\, K
\longrightarrow Z^s \longrightarrow Z^s/G(\epsilon)^{ad}$. We want to show that 
for every 
$\tau \in
\mbox{Gal}(K/F)$, the diagram

\begin{equation}\label{eqn-1}
 \begin{CD}
  \mbox{Spec }K		@>\pi \circ \phi>>	Z^s/G(\epsilon)^{ad} \\
  @VV\tau V					@|	   \\
  \mbox{Spec }K		@>\pi \circ \phi>>	Z^s/G(\epsilon)^{ad}
 \end{CD}
\end{equation}
commutes and hence the map 
$\pi \circ \phi: \mbox{Spec }K \longrightarrow Z^s/G(\epsilon)^{ad}$ descends to 
give a map
$\psi:\mbox{Spec }F \longrightarrow Z^s/G(\epsilon)^{ad}$.

Let $\{ A_{\tau} \}$ be a 1-cocycle representing the class 
$$[(A,\sigma)] \in H^1(\mbox{Gal}(K/F),G(\epsilon)^{ad}(K)).$$
 Recall that this
defines an action of $\mbox{Gal}(K/F)$ on $A_K$, and $A \subseteq A_K$ consists
of the invariant elements. Since the elements $x_i$ are in $A$, they are
invariant under the $\mbox{Gal}(K/F)$-action. Hence, we have
\[
 x_i = A_{\tau} (x_i)^{\tau} (A_{\tau})^{-1}  
\]
for all $\tau \in \mbox{Gal}(K/F)$. Translating this into a commutative diagram, we get
\begin{equation}\label{eqn-3}
 \begin{CD}
  \mbox{Spec }K 	@>\phi>> 	Z^s \\
  @V\tau VV					@VV A_{\tau} (-) A_{\tau}^{-1}V \\
  \mbox{Spec }K		@>\phi>>	Z^s,
 \end{CD}
\end{equation}
and composing with $\pi:Z^s \longrightarrow Z^s/G(\epsilon)^{ad}$ gives us 
exactly the diagram in
\eqref{eqn-1}. Hence we obtain the desired map $\psi:\mbox{Spec 
}F\longrightarrow Z^s/G(\epsilon)^{ad}$.

We are now going to prove that $\psi$ is 
independent of the choice of the finite Galois extension $K/F$ splitting 
$(A,\sigma,x_i)$ as well as the choice of the isomorphism $(A,\sigma)_K 
\tildearrow \mbox{End}(K^{2m},Q_0)$. Let 
$K_1/F$ and $K_2/F$ be two such extensions, and let $K/F$ be a finite Galois 
extension containing $K_1$ and $K_2$. Choosing two isomorphisms $A_{K_i} 
\tildearrow \mbox{End}(K_{i}^{2m},Q_0)$, $i=1,2$, and extending them to $K$,
we obtain two maps $\phi_i: \mbox{Spec }K_i\longrightarrow Z^s$ whose 
compositions with 
$j_i: \mbox{Spec }K\longrightarrow \mbox{Spec }K_i$ differ by an element of 
$G(\epsilon)^{ad}(K)$. Hence we have that $\pi \circ \phi_1 \circ j_1 = \pi 
\circ \phi_2 \circ j_2$. 

Now $\pi \circ \phi_1$ descends to $\psi_1: \mbox{Spec }F\longrightarrow 
Z^s/G(\epsilon)^{ad}$ 
as proved above. Hence $\pi \circ \phi_1 \circ j_1$ also descends to $\psi_1$. 
Similarly, $\pi \circ \phi_2$ descends to $\psi_2: {\rm Spec}\,F\longrightarrow 
Z^s/G(\epsilon)^{ad}$, and hence $\pi \circ \phi_2 \circ j_2$ descends to 
$\psi_2$. Since $\pi \circ \phi_1 \circ j_1 = \pi \circ \phi_2 \circ j_2$, it follows that $\psi_1 = \psi_2$. This finishes the construction of $\beta$.

It remains to prove that $\alpha$ and $\beta$ are inverses to each other. To 
prove that $\alpha \circ \beta={\rm Id}$, consider a stable CSA with involution 
$(A,\sigma,x_i)$ over a field $F$, and choose a finite Galois extension $K/F$ 
splitting it. We obtain a morphism $\phi: {\rm Spec }\,K \longrightarrow Z^s$ 
whose 
composition with $\pi$ descends to $\psi:\mbox{Spec }F \longrightarrow 
Z^s/G(\epsilon)^{ad}$. 
We note that $\pi \circ \phi$ pulls back $\mathcal{B}$ to $(A,\sigma,x_i)_K$.

Consider the class $[(A,\sigma)] \in H^1(\mbox{Gal}(K/F),G(\epsilon)^{ad}(K))$. 
Let $\{ A_{\tau} \}$ be a 1-cocycle representing this class. Since the $x_i$ 
come from elements in $A$, we have $x_i=A_{\tau} (x_i)^{\tau} A_{\tau}^{-1}$; in 
other words, we have the commutative diagram in \eqref{eqn-3}. Hence the action 
of 
$\mbox{Gal}(K/F)$ on $(A,\sigma,x_i)_K$ is the same as the action defined by the 
1-cocycle $\{ A_{\tau} \}$. This implies that $\mathcal{B}$ pulls back to 
$(A,\sigma,x_i)$ under $\psi:\mbox{Spec }F \longrightarrow 
Z^s/G(\epsilon)^{ad}$, and hence that $\alpha \circ \beta={\rm Id}$.

To prove that $\beta \circ \alpha={\rm Id}$, take $\psi:\mbox{Spec }F 
\longrightarrow 
Z^s/G(\epsilon)^{ad}$. Consider the stable CSA with involution obtained by 
pulling $\mathcal{B}$ to $\mbox{Spec }F$ via $\psi$. Take a finite Galois 
extension $K/F$ splitting $\psi^* \mathcal{B}$, and obtain a morphism 
$\phi:\mbox{Spec }K \longrightarrow Z^s$. We then need to prove that the diagram
\begin{equation*}
 \begin{CD}
  \mbox{Spec }K		@>\phi>>	Z^s \\
  @VVV					@VV\pi V	   \\
  \mbox{Spec }F		@>\psi>>	Z^s/G(\epsilon)^{ad}
 \end{CD}
\end{equation*}
commutes. (Note that there is a slight ambiguity in the choice of $\phi$ here, 
due to the fact that one must choose an isomorphism $(\psi^* \mathcal{B})_K 
\longrightarrow 
\mbox{End}(K^{2m},Q_0)$. However, since such choices only affect $\phi$ up to conjugation by an element of $G(\epsilon)^{ad}$, this will not be a problem.)

This follows from the following claim: Any morphism $f:\mbox{Spec }M 
\longrightarrow 
Z^s/G(\epsilon)^{ad}$ that pulls $\mathcal{B}$ back to a \emph{split} stable CSA 
with involution, lifts to a morphism $g:\mbox{Spec }M \longrightarrow Z^s$.

\emph{Proof of the claim: } Consider a morphism $f:\mbox{Spec }M \longrightarrow 
Z^s/G(\epsilon)^{ad}$ that pulls $\mathcal{B}$ back to a split stable CSA with 
involution, giving a morphism $g:\mbox{Spec }M \longrightarrow Z^s$ that pulls 
$\mathcal{A}$ 
back to the stable CSA with involution on $\mbox{Spec }M$. Let $\overline{M}$ 
denote the algebraic closure of $M$. The composition of $f$ with $\mbox{Spec 
}\overline{M}\longrightarrow \mbox{Spec }M$ lifts to $Z^s$, i.e., the outer 
square in the diagram
\[
 \xymatrix{
 \mbox{Spec }\overline{M} \ar[d]_i \ar[r]^{\overline{f}} & Z^s \ar[d]^{\pi} \\
 \mbox{Spec }M \ar[ur]^g \ar[r]^f & Z^s/G(\epsilon)^{ad}
 }
\]
commutes. Now it is easily checked that $g \circ i$ and $\overline{f}$ pull the stable CSA with involution $\mathcal{A}$ back to isomorphic stable CSAs with involution over $\overline{M}$. Hence $g \circ i = \overline{f}$; and this implies that $f=\pi \circ g$, finishing the proof of the claim and the proof of the theorem.
\end{proof}

\section{Construction of  Stable CSAs with
involution}\label{sec-construction}

Throughout this section, we will write $\epsilon=2^\alpha s$. We also remind
the 
reader that $g \geq 2$, where $g$ was defined in Remark \ref{rem-Z}. 

Given a field $L$ and two elements $\alpha, \beta\in L\setminus L^{\times 2}$ we
denote by
$(\alpha,\beta)$ the quaternion algebra formed by taking square roots of 
$\alpha$ and $\beta$. Concretely, this is the subalgebra of
$$ 
M_2(L(\sqrt{\alpha}, \sqrt{\beta}))
$$
generated by the matrices 
$$
\left(
\begin{array}{cc}
 \sqrt{\alpha} & 0 \\
  0 & -\sqrt{\alpha}
\end{array}
\right) \quad
\left(
\begin{array}{cc}
  0 & 1 \\
  \beta & 0
\end{array}
\right).
$$
The reader is referred to Appendix \ref{sec:csa-inv}
for more details.

Let 
$F=k(x_1, \cdots,x_{\alpha},y_1,\cdots,y_{\alpha})$ and 
$K=k(\sqrt{x_1},\cdots,\sqrt{x_{\alpha}},y_1,\cdots,y_{\alpha})$.

We let 
$$
D = (x_1,y_1) \otimes \cdots \otimes (x_{\alpha},y_{\alpha})
$$

\begin{thm} \label{t:amitsur-index}
 The central simple algebra $D$
over $F$ is a division algebra and hence has index $2^\alpha$.
\end{thm}

\begin{proof}
 See \cite[Theorem 3]{sA72}.
\end{proof}

\subsection{The symplectic case with $\alpha$ odd} 

Recall from the appendix B that the quaternion algebras $(x_i,y_i)$ have a
natural symplectic involution that we
denoted by $\sigma_i$.

Suppose that $\alpha$ is odd and recall that $\epsilon=2^\alpha s$.
The central simple algebra $D\otimes M_s(F)$
has a symplectic involution
$$
\sigma_1\otimes\sigma_2\otimes\ldots\sigma_\alpha\otimes t
$$
where $t$ is the transpose involution on the matrix algebra. (See \cite[Proposition 2.23]{KMRT98}.)

\begin{prop}\label{prop-sp1}
Suppose that $\alpha$ is odd.
Denote by $i_n$ (respectively, $j_n$) the square roots of
$x_n$ (respectively, $y_n$) in the algebra 
$$
D \otimes M_s(F).
$$
If
\begin{eqnarray*}
 A & = & i_1\otimes i_2\otimes\cdots\otimes i_\alpha\otimes{\rm 
diag}(1,2,\cdots,s) \\
 B & = & j_1\otimes j_2\otimes\cdots\otimes j_\alpha\otimes I_s
\end{eqnarray*}
then the collection of elements
$\lambda_1 =A$, $\lambda_2=x_3=\ldots = \lambda_g=B$ gives $D$ the structure of
a stable
central simple algebra with symplectic involution.
\end{prop}
\begin{proof}
In what follows, we will think of $D$ as a subalgebra of a matrix algebra over
$K$ in the usual way.

 It is easily checked that $\sigma(A)=-A$ and
$\sigma(B)=-B$.

Now consider standard bases $\{e_1,e_2\}, \cdots, 
\{e_{2\alpha-1},e_{2\alpha}\}, \{f_1,\cdots,f_s\}$ of the vector spaces $K^2$ 
and $K^s$. With respect to these standard bases, the eigenvalues of $A$ are 
$$\pm \sqrt{x_1 \cdots x_{\alpha}}, \cdots, \pm s\sqrt{x_1 \cdots 
x_{\alpha}}\, .$$
Consider the eigenvalue $\sqrt{x_1 \cdots x_{\alpha}}$. (The proof for the other
eigenvalues is similar and shall be omitted.) The eigenspace is spanned by the
vectors $e_{i_1} \otimes \cdots e_{i_{\alpha}} \otimes f_1$, where $i_k$ is
either $2k-1$ or $2k$ and an even number of the $i_k$'s are even. For
simplicity, we denote $e_{i_1} \otimes \cdots e_{i_{\alpha}} \otimes f_1$ by
$e_{i_1,\cdots,i_{\alpha}}$.

Standard arguments show that if $v\in M$ has a non-zero projection onto an
eigenspace then $M$ must contain a non-zero eigenvector for that eigenspace.

Consider a nonzero $\Lambda \otimes K$-submodule $M$ of $K^{\epsilon}$. Take a 
nonzero vector $$v=\sum \lambda_{i_1,\cdots,i_{\alpha}} 
e_{i_1,\cdots,i_{\alpha}} \in M\, , ~\, \lambda_{i_1,\cdots,i_{\alpha}} \in K$$ 
that 
is an eigenvector for $\sqrt{x_1 \cdots x_{\alpha}}$. Then we have
\begin{equation*}
 Bv=\sum \lambda_{i_1,\cdots,i_{\alpha}} y_1^{b_1} \cdots 
y_{\alpha}^{b_{\alpha}} e_{\overline{i_1},\cdots,\overline{i_{\alpha}}}\, ,
\end{equation*}
where $\overline{i_k}=2k-1$ if $i_k=2k$ and $\overline{i_k}=2k$ if $i_k=2k-1$.
Also, $b_k=0$ if $i_k$ is odd and $b_k=1$ is $i_k$ is even.\footnotetext[1]{We
keep this notation throughout the rest of the section.} Hence an even number of
the $b_k$'s are 1.

The symplectic involution $\sigma$ is adjoint with respect to the symplectic 
form given by $Q=\Sigma \otimes \cdots \Sigma \otimes I_s$. We claim that
$Q(v,Bv) \neq 0$, which proves that the $\Lambda \otimes K$-submodule $M$ of
$K^{\epsilon}$ is not isotropic, and that we have a stable CSA with involution
$(D \otimes M_s(F),\sigma,\{x_1,\cdots,x_g\})$.

For each term $\lambda_{i_1,\cdots,i_{\alpha}} e_{i_1,\cdots,i_{\alpha}}$ in 
$v$, the only term in $Bv$ for which $Q(-,-)$ is nonzero is 
$\lambda_{i_1,\cdots,i_{\alpha}} y_1^{b_1} \cdots y_{\alpha}^{b_{\alpha}} 
e_{\overline{i_1},\cdots,\overline{i_{\alpha}}}$. Hence, we have
\begin{equation*}
 Q(v,Bv)=\sum \lambda_{i_1,\cdots,i_{\alpha}}^2 y_1^{b_1} \cdots 
y_{\alpha}^{b_{\alpha}}.
\end{equation*}
Assume that $Q(v,Bv)=0$. Then we have the equation
\begin{equation*}
 \sum \lambda_{i_1,\cdots,i_{\alpha}}^2 y_1^{b_1} \cdots 
y_{\alpha}^{b_{\alpha}}=0.
\end{equation*}
Using Lemma \ref{lem:pfister} we see that this is a contradiction, hence
finishing the proof.
\end{proof}

\subsection{The symplectic case with $\alpha$ even} 

The quaternion algebra $(x_\alpha,y_\alpha)$ has an orthogonal 
involution $\tau$, described in the appendix. The involution
$$
\sigma_1\otimes\sigma_2\otimes\cdots\otimes\sigma_{\alpha - 1}\otimes\tau\otimes 
t
$$
is a symplectic involution on $D\otimes M_s(F)$. (See \cite[Proposition 2.23]{KMRT98}.)

\begin{prop}\label{prop-sp2}
Suppose that $\alpha$ is even.
Denote by $i_n$ (respectively, $j_n$) the square roots of
$x_n$ (respectively, $y_n$) in the algebra 
$$
D \otimes M_s(F).
$$
If
\begin{eqnarray*}
 A & = & i_1\otimes i_2\otimes\cdots\otimes i_\alpha\otimes{\rm 
diag}(1,2,\cdots,s) \\
 B & = & j_1\otimes j_2\otimes\cdots\otimes j_{\alpha-1}\otimes 1 \otimes I_s
\end{eqnarray*}
then the collection of elements
$\lambda_1 =A$, $\lambda_2=\lambda_3=\ldots = \lambda_g=B$ gives $D$ the
structure of a stable
central simple algebra with symplectic involution.

\end{prop}
\begin{proof}
It is easily checked that $\sigma(A)=-A$ and
$\sigma(B)=-B$.

The eigenvalues of $A$ are as in the proof of Proposition \ref{prop-sp1}, and we keep 
the notation. Consider a nonzero $\Lambda \otimes K$-submodule $M$ of 
$K^{\epsilon}$. Take a nonzero vector $$v=\sum \lambda_{i_1,\cdots,i_{\alpha}} 
e_{i_1,\cdots,i_{\alpha}} \in M\, ,~\,\lambda_{i_1,\cdots,i_{\alpha}} \in K\, 
,$$ that
is an eigenvector for $\sqrt{x_1 \cdots x_{\alpha}}$. (The cases of the other
eigenvalues are similar and are left to the reader.) Then we have

\begin{equation*}
 Bv=\sum \lambda_{i_1,\cdots,i_{\alpha}} y_1^{b_1} \cdots 
y_{\alpha-1}^{b_{\alpha-1}} 
e_{\overline{i_1},\cdots,\overline{i_{\alpha-1}},i_{\alpha}}.
\end{equation*}

The symplectic involution $\sigma$ is adjoint with respect to the symplectic 
form given by $Q=\Sigma \otimes \cdots \Sigma \otimes T \otimes I_s$. We claim
that $Q(v,Bv) \neq 0$, which proves that the $\Lambda \otimes K$-submodule $M$
of $K^{2m}$ is not isotropic, and that we have a stable CSA with involution $(D
\otimes M_s(F),\sigma,\{x_1,\cdots,x_g\})$.

For each term $\lambda_{i_1,\cdots,i_{\alpha}} e_{i_1,\cdots,i_{\alpha}}$ in 
$v$, the only term in $Bv$ for which $Q(-,-)$ is nonzero is 
$\lambda_{i_1,\cdots,i_{\alpha}} y_1^{b_1} \cdots y_{\alpha-1}^{b_{\alpha-1}} 
e_{\overline{i_1},\cdots,\overline{i_{\alpha-1}},i_{\alpha}}$. Hence, we have
\begin{align*}
 &Q(v,Bv)= \\
 &\sum_{\parbox{3in}{$i_{\alpha}$ is odd,\\ an even number of 
$i_1,\cdots,i_{\alpha-1}$ are even}} \lambda_{i_1,\cdots,i_{\alpha}}^2 
y_1^{b_1} \cdots y_{\alpha-1}^{b_{\alpha-1}} \frac{1}{\sqrt{x_{\alpha}}} \\
 \pm &\sum_{\parbox{3in}{$i_{\alpha}$ is even,\\ an odd number of 
$i_1,\cdots,i_{\alpha-1}$ are even}} \lambda_{i_1,\cdots,i_{\alpha}}^2 
y_1^{b_1} \cdots y_{\alpha-1}^{b_{\alpha-1}}
\frac{1}{y_{\alpha}\sqrt{x_{\alpha}}}.
\end{align*}

Above, the $\pm$ sign is determined by the parity of the $i_k$'s. Assume
$Q(v,Bv)=0$. Then we have
\begin{align*}
 & \sum_{\parbox{3in}{$i_{\alpha}$ is odd,\\ an even number of 
$i_1,\cdots,i_{\alpha-1}$ are even}} \lambda_{i_1,\cdots,i_{\alpha}}^2 
y_1^{b_1} \cdots y_{\alpha-1}^{b_{\alpha-1}} \\
 = & \sum_{\parbox{3in}{$i_{\alpha}$ is even,\\ an odd number of 
$i_1,\cdots,i_{\alpha-1}$ are even}} \lambda_{i_1,\cdots,i_{\alpha}}^2 
y_1^{b_1} \cdots y_{\alpha-1}^{b_{\alpha-1}} \frac{1}{y_{\alpha}}.
\end{align*}

(Above, we incorporate the possible $-$ sign into the 
$\lambda_{i_1,\cdots,i_{\alpha}}^2$ since the base field $k$ contains a square
root of $-1$.) Multiplying both sides by $y_{\alpha}$, and using Lemma
\ref{lem:pfister} as before we see that this is a contradiction, hence finishing
the proof.
\end{proof}

\subsection{The orthogonal case with $\alpha$ odd and $s\ne 1$}

Recall the definition of the involution $\delta$ on a quaternion algebra 
from the appendix.
The involution 
$$\sigma=\delta_1 \otimes \cdots \otimes \delta_{\alpha}
\otimes t$$
 on $D \otimes M_s(F)$  is orthogonal by
\cite[Proposition 2.23]{KMRT98}.

Consider the following elements of $D\otimes M_s(F)$ :
\begin{align*}
 A &= i_1 \otimes \cdots \otimes i_{\alpha} \otimes \mbox{diag}(1,\cdots,s) \\
 B &= (i_1 \otimes \cdots \otimes i_{\alpha} \otimes M_1) + (j_1 \otimes \cdots
\otimes j_{\alpha} \otimes M_2)
\end{align*}
of $D \otimes M_s(F)$, where we have $M_1, M_2 \in M_s(F)$ defined as
\[
 M_1 = \left(
        \begin{matrix}
          1 & 1 & \cdots & 1 \\
          1 & 0 & \cdots & 0 \\
          \vdots & \vdots & \ddots & \vdots \\
          1 & 0 & \cdots & 0
        \end{matrix}
       \right),
\]
and
\[
 M_2 = \left(
        \begin{matrix}
          0 & 1 & \cdots & 1 \\
         -1 & 0 & \cdots & 0 \\
          \vdots & \vdots & \ddots & \vdots \\
          -1 & 0 & \cdots & 0
        \end{matrix}
       \right).
\]

\begin{prop}\label{prop-so1}
 Suppose $\alpha$ odd and $s\ne 1$.
In the above notation the system of elements
$\lambda_1=A$ and $\lambda_2=\lambda_3=\cdots \lambda_g=B$ gives 
$D\otimes M_s(F)$ the structure of a stable
central simple algebra with orthogonal involution.
\end{prop}
\begin{proof}

It is easily checked that $\sigma(A)=-A$ and $\sigma(B)=-B$.

The eigenvalues of $A$ are as in the proof of Proposition \ref{prop-sp1}, and we keep 
the notation. Consider a nonzero $\Lambda \otimes K$-submodule $M$ of $K^{2m}$. 
Take a nonzero vector $v=\sum \lambda_{i_1,\cdots,i_{\alpha}} 
e_{i_1,\cdots,i_{\alpha}} \in M$, $\lambda_{i_1,\cdots,i_{\alpha}} \in K$, that
is an eigenvector for $\sqrt{x_1 \cdots x_{\alpha}}$. (The cases of the other
eigenvalues are similar and are left to the reader.) Then we have

\begin{align*}
 Bv &= \sum \lambda_{i_1,\cdots,i_{\alpha}} (\sqrt{x_1 \cdots x_{\alpha}} 
e_{i_1} \otimes \cdots e_{i_{\alpha}} \otimes (f_1+\cdots f_s) \\
    &+ y_1^{b_1} \cdots y_{\alpha}^{b_{\alpha}} e_{\overline{i_1}} \otimes
\cdots \otimes e_{\overline{i_{\alpha}}} \otimes (-f_2 - \cdots - f_s)).
\end{align*}

The orthogonal involution $\sigma$ is adjoint with respect to the symplectic 
form given by $Q=\Delta \otimes \cdots \otimes\Delta \otimes I_s$. We claim that
$Q(Bv,Bv) \neq 0$, which proves that the $\Lambda \otimes K$-submodule $M$ of
$K^{2m}$ is not isotropic, and that we have a stable CSA with involution $(D
\otimes M_s(F),\sigma,\{x_1,\cdots,x_g\})$.

Indeed, one computes
\[
 Q(Bv,Bv)=2(-s+1)\sum \lambda_{i_1,\cdots,i_{\alpha}}^2 y_1^{b_1} \cdots 
y_{\alpha}^{b_{\alpha}} \sqrt{x_1 \cdots x_{\alpha}}.
\]

By assumption, $-s+1 \neq 0$. Hence if one assumes that $Q(Bv,Bv)=0$, one obtains 
a contradiction using Lemma \ref{lem:pfister}. This finishes the proof.
\end{proof}

\subsection{The orthogonal case with $\alpha$ even and $s\ne 1$}

The involution 
$$\sigma=\delta_1 \otimes \cdots \otimes \delta_{\alpha-1}
\otimes t_{\alpha} \otimes t
$$
 on $D \otimes M_s(F)$  is orthogonal by
\cite[Proposition 2.23]{KMRT98}.

Define two
elements
\begin{align*}
 A &= i_1 \otimes \cdots \otimes i_{\alpha} \otimes \mbox{diag}(1,\cdots,s) \\
 B &= (i_1 \otimes \cdots \otimes i_{\alpha} \otimes M_1) + (j_1 \otimes \cdots
\otimes j_{\alpha-1} \otimes 1 \otimes M_2)
\end{align*}
of $D \otimes M_s(F)$, where $M_1$ and $M_2$ are as in the previous 
subsection.

\begin{prop}\label{prop-so2}
 Suppose $\alpha$ even and $s\ne 1$.
In the above notation the system of elements
$\lambda_1=A$ and $\lambda_2=\lambda_3=\cdots \lambda_g=B$ gives 
$D\otimes M_s(F)$ the structure of a stable
central simple algebra with orthogonal involution.
\end{prop}

\begin{proof}
It is easily
checked that $\sigma(A)=-A$ and $\sigma(B)=-B$.

The eigenvalues of $A$ are as in the proof of Proposition \ref{prop-sp1}, and we keep 
the notation. Consider a nonzero $\Lambda \otimes K$-submodule $M$ of $K^{2m}$. 
Take a nonzero vector $$v=\sum \lambda_{i_1,\cdots,i_{\alpha}} 
e_{i_1,\cdots,i_{\alpha}} \in M\, ,~\, \lambda_{i_1,\cdots,i_{\alpha}} \in K\, 
,$$ that is 
an eigenvector for $\sqrt{x_1 \cdots x_{\alpha}}$. (The cases of the other eigenvalues are similar and are left to the reader.) Then we have

\begin{align*}
 Bv &= \sum \lambda_{i_1,\cdots,i_{\alpha}} (\sqrt{x_1 \cdots x_{\alpha}} 
e_{i_1} \otimes \cdots e_{i_{\alpha}} \otimes (f_1+\cdots f_s) \\
    &+ y_1^{b_1} \cdots y_{\alpha}^{b_{\alpha}} e_{\overline{i_1}} \otimes
\cdots \otimes e_{\overline{i_{\alpha-1}}} \otimes e_{i_{\alpha}} \otimes (-f_2
- \cdots - f_s)).
\end{align*}

The orthogonal involution $\sigma$ is adjoint with respect to the symplectic 
form given by $Q=\Delta^{\otimes \alpha-1} \otimes I_2 \otimes I_s$. We claim
that $Q(Bv,Bv) \neq 0$, which proves that the $\Lambda \otimes K$-submodule $M$
of $K^{2m}$ is not isotropic, and that we have a stable CSA with involution $(D
\otimes M_m(F),\sigma,\{x_1,\cdots,x_g\})$.

Indeed, one computes
\[
 Q(Bv,Bv)=2(-s+1)\sum \lambda_{i_1,\cdots,i_{\alpha}}^2 y_1^{b_1} \cdots 
y_{\alpha}^{b_{\alpha}} \sqrt{x_1 \cdots x_{\alpha}}\, .
\]

By assumption, $-s+1 \neq 0$. Hence if one assumes $Q(Bv,Bv)=0$, one obtains a
contradiction using Lemma \ref{lem:pfister}. This finishes the proof.
\end{proof}

\begin{rem}\label{rem-so}
In the proofs of Propositions \ref{prop-so1} and \ref{prop-so2}, if one 
replaces the factor $(x_{\alpha},y_{\alpha})$ in $D$ by $M_2(F)$
and $\delta_{\alpha}$) $t_{\alpha}$, while $F$ and $K$ are changed so that they 
have
$\alpha-1$ number of $x$ and $y$ variables; the same proofs carry through and 
hence there
exists a stable CSA with involution $( (x_1,y_1) \otimes \cdots \otimes
(x_{\alpha-1},y_{\alpha-1}) \otimes M_2(F), \sigma, \{x_1,\cdots,x_g\} )$ over
$F$.
\end{rem}

\section{The Main results}\label{sec-results}

Recall that  $\epsilon=2^\alpha s$ with $s$ odd. We also remind the 
reader that $g \geq 2$, where $g$ was defined in Remark \ref{rem-Z}. Recall that
$F=k(x_1, \cdots,x_{\alpha},y_1,\cdots,y_{\alpha})$ and 
$K=k(\sqrt{x_1},\cdots,\sqrt{x_{\alpha}},y_1,\cdots,y_{\alpha})$. 

\begin{thm}\label{thm:amitsurs-algebra}
\mbox{}
\begin{enumerate}
\item For $G(\epsilon)={\rm Sp}(\epsilon)$ or $G(\epsilon)={\rm SO}(\epsilon)$
with $s > 1$, there 
exists a stable CSA with involution $(D \otimes
M_s(F),\sigma,\{x_1,\cdots,x_g\})$ over $F$. 
\item If $G(\epsilon)={\rm SO}(\epsilon)$ with $\epsilon=2^\alpha$, then there
exists a stable CSA with involution $( (x_1,y_1) \otimes \cdots \otimes
(x_{\alpha-1},y_{\alpha-1}) \otimes M_2(F), \sigma, \{x_1,\cdots,x_g\} )$ over
$F$.
\end{enumerate}
\end{thm}

\begin{proof}
See Propositions \ref{prop-sp1}, \ref{prop-sp2}, \ref{prop-so1}, \ref{prop-so2}
and Remark \ref{rem-so}.
\end{proof}

\begin{cor}\label{cor-ind}{\ }
Recall the definition of the CSA $\mathcal{B}$ from Example \ref{e:triple}.
\begin{enumerate}
 \item For $G(\epsilon)={\rm Sp}(\epsilon)$ or
$G(\epsilon)={\rm SO}(\epsilon)$ 
with $s > 1$, the index of $\mathcal{B}$ is divisible by $2^\alpha$.
 \item For $G(\epsilon)={\rm SO}(\epsilon)$ with $\epsilon=2^\alpha$, 
the index of $\mathcal{B}$ is divisible by $2^{\alpha-1}$.
\end{enumerate}
\end{cor}

\begin{proof}
 Combine Theorem \ref{thm-rep} with Theorem \ref{t:amitsur-index} and Theorem
\ref{thm:amitsurs-algebra}.
\end{proof}

Let $\widehat{F}=k((x_1,y_1,\cdots,x_{\alpha},y_{\alpha}))$ and 
$\widehat{K}=k((\sqrt{x_1},y_1,\cdots,\sqrt{x_{\alpha}},y_{\alpha}))$. (In the 
case 
of $G(\epsilon)={\rm SO}(\epsilon)$ with $\epsilon=2^\alpha$, 
it is understood that there would be $\alpha-1$ number of $x$ and $y$ variables 
in the
definition. We will assume this tacitly to avoid repetition in the below proof.)
Let $\psi_0$ denote the maps ${\rm Spec}\, F\longrightarrow 
Z^s/G(\epsilon)^{ad}$
constructed in the proof of Theorem \ref{thm:amitsurs-algebra}. Composing
$\psi_0$ with the canonical map ${\rm Spec}(\widehat{F})\longrightarrow 
\mbox{Spec}(F)$, we
obtain a map 
$$\phi_0:\mbox{Spec}(\widehat{F})\longrightarrow Z^s/G(\epsilon)^{ad}\, .$$
 We have
$$\phi_0^*(\mathcal{A}) \cong (x_1,y_1) \otimes \cdots \otimes
(x_{\alpha},y_{\alpha}) \otimes M_s(\widehat{F})\, .$$

Recall the open subscheme
$$\spec(\widehat{\struct_0})^t $$
constructed in Section 2. 
 
\begin{prop}\label{prop-factor}
We have a factorization :
\[
 \xymatrix{
 \spec(\widehat{F}) \ar[d] \ar[dr] & \\
  \spec(\widehat{\struct_0})^t \ar[r]& Z^s/G(\epsilon)^{ad}.
}
\]
\end{prop}

\begin{proof}
Firstly the map $\psi_0$ factors through $(Z/\!\!/G)^t$ as it factors
through the stable locus as $(Z/\!\!/G)^t$ is dense in $(Z/\!\!/G)^s$, being 
an open subset of an irreducible set. So it suffices to show that
$\phi_0$ factor through the completion of the local ring at $0$.

The remainder of the proof is essentially the same as the proof of Corollary
6.2 in \cite{CDL10}.

We take $R=k[\sqrt{x_1},\cdots,\sqrt{x_{\alpha}},y_1,\cdots,y_{\alpha}]$.
As the formulas for the $\lambda_i$ in Propositions \ref{prop-sp1},
\ref{prop-sp2}, \ref{prop-so1} and \ref{prop-so2} do not involve
denominators there is a
diagram
\[
 \xymatrix{
 \spec(R)        \ar[dr]^\Phi       & \\
 \spec(K) \ar[u] \ar[d] & Z  \ar[d]\\
 \spec(F)         \ar[r]      & Z/\!\!/G(\epsilon)
}
\]

Denote $\Phi(0)=P \in Z$. Hence we get a map on completions of local rings, and we obtain a map
\begin{equation*}
 \phi_c:{\rm Spec}(\widehat{K}) \longrightarrow \widehat{\mathcal{O}_{Z,P}}.
\end{equation*}
In view of Theorem \ref{thm:amitsurs-algebra}, it is enough to show that it 
descends to a morphism ${\rm Spec}(\widehat{F}) \longrightarrow
\widehat{\mathcal{O}_{Z/\!\!/G,0}^{s}}$. We have the following commutative 
diagram
\[
 \begin{CD}
  \mathcal{O}_{Z/\!\!/G(\epsilon),0} @>>> \widehat{F} \\
  @VVV @VVV \\
  \mathcal{O}_{Z,P} @>>> \widehat{K}
 \end{CD}
\]
with $\phi_c$ induced by completion from the bottom line. The image of the 
maximal ideal of $\mathcal{O}_{Z/\!\!/G(\epsilon),0}$ must be contained in the 
$\mathcal{O}_{Z/\!\!/G(\epsilon),0}$ submodule of $\widehat{F}$ generated by 
$(x_1,y_1,\cdots,x_{\alpha},y_{\alpha})$. But the field $\widehat{F}$ is 
complete with respect to the induced topology and hence we obtain our map.
\end{proof}

\begin{thm}[Index of the canonical gerbe]{\ }
\label{t:theindex}
\begin{enumerate}
 \item For $G(\epsilon)={\rm Sp}(2n)$ or $G(\epsilon)={\rm SO}(2n)$ with $s > 
1$, the 
index of 
the 
canonical gerbe $\text{Bun}^{\rm rs}_{G} \longrightarrow M^{\rm rs}_{G}$ is 
$2^{\alpha}$.
 \item For $G(\epsilon)={\rm SO}(2n)$ with $\epsilon=2^\alpha$, the index of the 
canonical gerbe $\text{Bun}^{\rm rs}_{G} \longrightarrow M^{\rm rs}_{G}$ is 
$2^{\alpha-1}$ or $2^{\alpha}$.
\end{enumerate}
\end{thm}

\begin{proof}
By Propositions \ref{p:same} and \ref{prop-factor}, the 
index of the canonical gerbe is divisible by $2^\alpha$ (or $2^{\alpha-1}$ in the second case). By Proposition \ref{prop-indexdivides}, the index of the canonical gerbe divides $\epsilon$ and hence it divides $2^\alpha$. The result now follows from Corollary \ref{cor-ind}.
\end{proof}

\appendix

\section{An Anisotropic Form over the Field of Rational Functions}
The following well-known result was used in the construction of stable CSAs with involution.
\begin{lem}\label{lem:pfister}
The $n$-Pfister form $\ll t_1,\cdots,t_n \gg$ over the field 
$k(t_1,\cdots,t_n)$ is anisotropic.
\end{lem}
\begin{proof}
We use induction on $n$. For $n=1$, the form $\ll t_1 \gg=\langle 1,t_1 \rangle$ has no isotropic vectors as the equation
\[
 f_1^2+t_1 f_2^2=0
\]
implies that $t_1$ is a square in $k(t_1,\cdots,t_n)$, a contradiction.

Assume that the statement is proved for $(n-1)$-Pfister forms, and consider the 
$n$-Pfister form $\ll t_1,\cdots,t_n \gg$. Assume that there exists an 
isotropic vector for $\ll t_1,\cdots,t_n \gg$, hence an equation
\[
 \sum_{I \subset \{ 1,\cdots,n \} } t_I f_I^2=0\, ,
\]
where $I$ runs over all subsets of $\{1,\cdots,n\}$, $t_I$ is the monomial 
obtained by multiplying the $t_i$ for which $i \in I$ and, by clearing denominators and removing common factors, we assume that the $f_I$ are polynomials with no common factors.

By setting $t_n=0$, we have $\sum_{I \subset \{ 1, \cdots, n-1 \} } t_I 
\overline{f_I}^2=0$, where $\overline{f_I}$ denotes the reduction of $f_I$ 
modulo $t_n$. By the induction hypothesis, this can only happen if all the 
$\overline{f_I}$ are zero; i.e., when $t_n$ divides the $f_I$ for $I \subset \{ 
1, \cdots, n-1 \}$.

Now rearrange the equation above to get
\[
 t_n^2 g = -t_n \sum_{I \subset \{ 1, \cdots, n-1 \} } t_{I \cup \{n\}} f_{I 
\cup \{n\}}^2
\]
for some polynomial $g$. After cancelling $t_n$ and setting $t_n=0$ again, we 
obtain a similar equation 
\[
\sum_{I \subset \{ 1, \cdots, n-1 \} } t_{I \cup \{n\}} \overline{f_{I \cup 
\{n\}}}^2=0\, ,
\]
which again implies that $t_n$ divides the remaining $f_I$. Hence $t_n$ divides all the $f_I$, which is a contradiction.
\end{proof}

\section{Central Simple Algebras with Involution}\label{sec:csa-inv}

Let $k$ be a field. Recall that a \emph{central simple algebra} over $k$ is a $k$-algebra $A$ which is finite-dimensional as a $k$-vector space, whose center is $k$ (viewed as a subring of $A$) and which has no proper, non-trivial two-sided ideals. Given a central simple algebra $A$ over $k$, there exists a finite field extension $L$ of $k$ such that $A_L=A \otimes_k L$ is isomorphic to a matrix algebra $M_n(L)$ over $L$. Hence, the dimension of $A$ over $k$ is a square $n^2$. The number $n$ is called the \emph{degree} of $A$.

Given two central simple algebras $A$ and $B$ over $k$, we call $A$ and $B$ 
\emph{Brauer-equivalent} if there exist natural numbers $m$ and $n$ such that 
$M_m(k) \otimes_k A \cong M_n(k) \otimes_k B$. The set of Brauer-equivalence 
classes of central simple algebras has the structure of an abelian group, 
denoted ${\rm Br}(k)$, described as follows. Multiplication of two elements 
$[A],[B] \in {\rm Br}(k)$ is given by $[A][B]=[A \otimes_k B]$, the identity 
element 
is given by the equivalence class $[k]$ of $k$ itself, and inverses are given by 
$[A]^{-1}=[A^\circ]$, where $A^\circ$ is the opposite algebra of $A$. 
${\rm Br}(k)$ 
is a torsion abelian group, and the order of an element of $a \in {\rm Br}(k)$ 
is 
called the \emph{period} of $a$.

By a theorem of Wedderburn, every central simple algebra $A$ over $k$ can be 
written as $M_n(D)$, where $D$ is a \emph{central division algebra} over $k$, 
meaning a central simple algebra over $k$ that is a division ring. $D$ is 
unique 
up to isomorphism. Hence, the degree of $D$ is well-defined, and is called the 
\emph{index} of $A$. Two Brauer-equivalent central simple algebras $A$ and $B$ 
over $k$ have the same index, hence the index is defined for elements of 
${\rm Br}(k)$.

There is a natural isomorphism
\[
{\rm Br}(K)\stackrel{\sym}{\longrightarrow} {\rm H}^2(K,\gm),
\]
and hence associated to every central simple algebra is a gerbe. The notion of index and period defined here agrees with the one in Section \ref{sec-twisted}. Hence by Proposition \ref{prop-periodindex}, we have that the period divides the index, and that the period and the index have the same prime powers.

An \emph{involution of the first kind} on a central simple algebra $A$ over $k$ 
is an additive map $\sigma: A \longrightarrow A$ such that 
$\sigma(xy)=\sigma(y)\sigma(x)$, 
$\sigma^2={\rm Id}_A$ and $\sigma(\lambda)=\lambda$ for all $\lambda \in k$. 
From now 
on, we will refer to an involution of the first kind as simply an involution.

Consider the central simple algebra $M_n(k)$ over $k$, which can also be viewed 
as $\mbox{End}_k(V)$, where $V$ is an $n$-dimensional vector space over $k$. 
Then there is a one-to-one correspondence between involutions on 
$\mbox{End}_k(V)$ and equivalence classes of nonsingular bilinear forms on $V$ 
modulo multiplication by an element of $k^\times$ that are either symmetric or 
skew-symmetric. (See the Theorem in the introduction to \cite[Chapter 
1]{KMRT98}.) Let $b$ be a symmetric or skew-symmetric bilinear form on $V$, and 
$\sigma$ the corresponding involution on $\mbox{End}_k(V)$. Fix an ordered basis 
for $V$ and denote the Gram matrix of $b$ with respect to this basis by $g \in 
{\rm GL}_n(k)$. Here, $g^t=g$ if $b$ is symmetric and $g^t=-g$ if $b$ is skew-symmetric. Then the involution $\sigma$ is given by
\begin{equation*}
 \sigma(m)=g^{-1}m^t g
\end{equation*}
for $m \in M_n(k)$.

Let $A$ be a central simple algebra over $k$, and $\sigma$ an involution on 
$A$. Choose a field extension $L$ of $k$ that splits $A$, i.e., 
$A_L\,=\,M_n(L)$. 
Over this base extension, consider the bilinear form $b$ that corresponds to the 
involution $\sigma_L = \sigma \otimes_k {\rm Id}_L$. If $b$ is symmetric, 
$\sigma$ is 
called \emph{orthogonal}; and if $b$ is skew-symmetric, $\sigma$ is called \emph{symplectic}.

Write $2n=2^{\alpha} m$ where $m$ is odd. Let 
$$F=\Bbbk(x_1,y_1,\cdots,x_{\alpha},y_{\alpha}) ~\, \text{ and }\,~ 
K=\Bbbk(\sqrt{x_1},y_1,\cdots,\sqrt{x_{\alpha}},y_{\alpha})\, .$$ Then $K/F$ is 
a Galois 
extension with Galois group isomorphic to $(\ZZ/2\ZZ)^{\alpha}$.

For $\ell=1,\cdots,\alpha$, let $(x_{\ell},y_{\ell})$ denote the quaternion 
algebra over $F$ having a basis $\{1,i,j,k\}$ such that $i^2=x_{\ell}$, 
$j^2=y_{\ell}$ and $k=ij=-ji$. The \emph{quaternion conjugation} or 
\emph{canonical involution} is the $F$-linear map 
$\sigma:(x_{\ell},y_{\ell})\longrightarrow 
(x_{\ell},y_{\ell})$ given by $a+bi+cj+dk \mapsto a-bi-cj-dk$. By \cite[Proposition 2.21]{KMRT98}, the canonical involution is the only symplectic involution on $(x_{\ell},y_{\ell})$. Note that over $K$, $(x_{\ell},y_{\ell})$ splits, and we have $\sigma_K=\mbox{Int}(\Sigma) \circ t$; where
\begin{equation*}
\Sigma=\left(
 \begin{matrix}
  0  &  1 \\
  -1 &  0
 \end{matrix}
\right),
\end{equation*}
and $t:M_2(K)\longrightarrow M_2(K)$ denotes the transpose involution. For later 
reference, 
we note that over $K$, $i_{\ell}$ and $j_{\ell}$ are represented by the matrices

\begin{equation*}
\left(
 \begin{matrix}
  \sqrt{x_{\ell}} & 0 \\
  0 & -\sqrt{x_{\ell}}
 \end{matrix}
\right),
\end{equation*}
and
\begin{equation*}
\left(
 \begin{matrix}
  0 & 1 \\
  y_{\ell} & 0
 \end{matrix}
\right).
\end{equation*}

We will also need two orthogonal involutions on $(x_{\ell},y_{\ell})$. Define $\tau=\mbox{Int}(k) \circ \sigma$ and $\delta=\mbox{Int}(i) \circ \sigma$. Then $\tau$ is given by $a+bi+cj+dk \mapsto a+bi+cj-dk$, $\delta$ is given by $a+bi+cj+dk \mapsto a-bi+cj+dk$ and they are orthogonal involutions by \cite[Proposition 2.21]{KMRT98}. After splitting $(x_{\ell},y_{\ell})$ by extending the base field to $K$, $\tau=\mbox{Int}(T) \circ t$ and $\delta=\mbox{Int}(\Delta) \circ t$, where
\begin{equation*}
T=\left(
 \begin{matrix}
  -\sqrt{x_{\ell}} & 0 \\
  0 & -y_{\ell} \sqrt{x_{\ell}}
 \end{matrix}
\right),
\end{equation*}
and
\begin{equation*}
\Delta=\left(
 \begin{matrix}
  0 & 1 \\
  1 & 0
 \end{matrix}
\right).
\end{equation*}

We note that on $M_m(K)$, the transpose involution $t$ is orthogonal.

The involutions $t \otimes_F {\rm Id}_K$, $\sigma \otimes_F {\rm Id}_K$, $\tau 
\otimes_F 
{\rm Id}_K$ and $\delta \otimes_F {\rm Id}_K$ on $(x_{\ell},y_{\ell}) \otimes_F 
K$ are 
the 
adjoint involutions with respect to the following forms: $t \otimes_F {\rm 
Id}_K$ is 
adjoint with respect to the orthogonal form represented with the identity matrix 
$I_m$, $\sigma \otimes_F {\rm Id}_K$ is adjoint with respect to the symplectic 
form represented by the matrix
\begin{equation*}
\Sigma^{-1}=\left(
 \begin{matrix}
  0 & -1 \\
  1 & 0
 \end{matrix}
\right),
\end{equation*}
$\tau \otimes_F {\rm Id}_K$ is adjoint with respect to the orthogonal form 
represented by the matrix
\begin{equation*}
T^{-1}=\left(
 \begin{matrix}
  -\frac{1}{\sqrt{x_{\ell}}} & 0 \\
  0 & -\frac{1}{y_{\ell} \sqrt{x_{\ell}}}
 \end{matrix}
\right),
\end{equation*}
and $\delta \otimes_F{\rm Id}_K$ is adjoint with respect to the orthogonal form 
represented by the matrix
\begin{equation*}
\Delta^{-1}=\left(
 \begin{matrix}
  0 & 1 \\
  1 & 0
 \end{matrix}
\right).
\end{equation*}

\end{document}